\documentclass[preprint,12pt]{elsarticle}

\usepackage{latexsym}
\usepackage{amsmath}
\usepackage{indentfirst}
\usepackage{amsthm}
\usepackage{amssymb}
\usepackage{amsfonts}
\usepackage{stackrel}
\usepackage{dsfont}
\usepackage{tikz}
\usepackage{slashbox}

\newcommand{\Boole}{\mathcal{B}}
\newcommand{\reals}{\mathbb{R}}
\newtheorem{thm}{Theorem}

\journal{Computers \& Mathematics with Applications}

\begin{document}
\begin{frontmatter}


\title{New error bounds for Boole's rule}
\author{Mateusz Krukowski}
\ead{krukowski.mateusz13@gmail.com}
\address{Technical University of Lodz, W\'olczanska 215, 90-924, \L\'od\'z}


\begin{abstract}
In recent years, a lot of research was devoted to Simpson's rule for numerical integration. In the paper we study a natural successor of Simpson's rule, namely the Boole's rule. It is the Newton-Cotes formula in the case where the interval of integration is divided into four subintervals of equal length. With computer software assistance, we prove novel error bounds for Boole's rule.
\end{abstract}


\end{frontmatter}


\section{Introduction}
\label{intro}

In recent years, a significant amount of research has been devoted to Simpson's rule for numerical integration:
$$\int_a^b\ f(t)\ dt \sim \frac{h}{3}\cdot \left(f(a) + 4f\left(\frac{a+b}{2}\right) + f(b)\right),\hspace{0.4cm} \text{where}\hspace{0.4cm} h= \frac{a+b}{2}.$$

\noindent
Among many formidable sources, let us mention just a few: \cite{DragomirAgarwalCerone, CruzUribeNeugebauer, Dragomir, DragomirPecaricWang, DragomirFedotov, Ujevic2004, Ujevic2007}. It is the author's impression that while Simpson's rule enjoys considerable popularity, Boole's rule remains somewhat neglected. The following paper is a modest attempt to change this status quo. 

The main part of the paper is Section \ref{mainresults}, in which we suggest six novel estimates for the Boole's rule
$$\int_a^b\ f(t)\ dt \sim \frac{2h}{45}\cdot \left(7f(a) + 32f\left(\frac{3a+b}{4}\right) + 12f\left(\frac{a+b}{2}\right) + 32f\left(\frac{a+3b}{4}\right)+ 7f(b)\right),$$

\noindent
where $h = \frac{a+b}{4}.$ For convenience, let us denote the right hand side of the Boole's rule by $\Boole(f)$.

Since the calculations become quite tedious very fast, we frequently resort to Maple computer software, which greatly facilitates computations. We took the liberty of enclosing some parts of the code in case the Readers wanted to verify the constants on their own.

Last but not least, in Section \ref{examples} we test the performance of the novel estimates. It turns out that for monomials $t\mapsto t^k$, (apart from first few instances), the new error bounds are better than the classical, well-know result.

\section{Main results}
\label{mainresults}

Before we proceed with the main results of the paper, let us denote 
$$I(g) = \frac{g(b)-g(a)}{b-a}.$$

\noindent
Furthermore, let $M(g)$ denote the constant $M$ such that 
$$\forall_{a.e.\ t\in[a,b]}\ g(t) \leq M,$$

\noindent
and similarly, let $m(g)$ denote the constant $m$ such that 
$$\forall_{a.e.\ t\in[a,b]}\ m\leq g(t).$$

\begin{thm}
Let $f:[a,b]\longrightarrow \reals$ be an absolutely continuous function such that $m(f')$ and $M(f')$ are both finite. Then 
\begin{gather}
\bigg|\int_a^b\ f(t)\ dt - \Boole(f)\bigg|\leq \frac{11}{60}\cdot \bigg(I(f)-m(f')\bigg) (b-a)^2
\label{mainresult11}
\end{gather}

\noindent
and
\begin{gather}
\bigg|\int_a^b\ f(t)\ dt - \Boole(f)\bigg|\leq \frac{11}{60}\cdot \bigg(M(f') - I(f)\bigg) (b-a)^2
\label{mainresult12}
\end{gather}
\end{thm}
\begin{proof}
Let us define 
$$K_i(t) = \alpha_it + \beta_i, \hspace{0.4cm}\text{for}\hspace{0.4cm} i=1,2,3,4,$$

\noindent
where $\alpha_i,\ \beta_i$ are constants which will be determined in the course of the proof. We put 
\begin{gather*}
K(t) = \left\{\begin{array}{ll}
K_1(t) & \text{if } t\in\left[a,\frac{3a+b}{4}\right],\\
K_2(t) & \text{if } t\in\left(\frac{3a+b}{4},\frac{a+b}{2}\right],\\
K_3(t) & \text{if } t\in\left(\frac{a+b}{2},\frac{a+3b}{4}\right],\\
K_4(t) & \text{if } t\in\left(\frac{a+3b}{4},b\right].
\end{array}\right.
\end{gather*}

Integration by parts yields
\begin{equation*}
\begin{split}
\int_a^b\ K(t)f'(t)\ dt = &-K_1(a)f(a) + \left(K_1\left(\frac{3a+b}{4}\right) - K_2\left(\frac{3a+b}{4}\right)\right)f\left(\frac{3a+b}{4}\right) \\
&+ \left(K_2\left(\frac{a+b}{2}\right) - K_3\left(\frac{a+b}{2}\right)\right)f\left(\frac{a+b}{2}\right) \\
&+ \left(K_3\left(\frac{a+3b}{4}\right) - K_4\left(\frac{a+3b}{4}\right)\right)f\left(\frac{a+3b}{4}\right) \\
&+ K_4(b)f(b) - \int_a^b\ K'(t)f(t)\ dt.
\end{split}
\end{equation*}

\noindent
We demand that $K'(t) \equiv 1$, so 
$$\alpha_1 = \alpha_2 = \alpha_3 = \alpha_4 = 1.$$

\noindent
Furthermore, we impose the following conditions:
\begin{gather*}
\left\{\begin{array}{l}
K_1(a) = -\frac{7}{90}\cdot(b-a),\\
K_1\left(\frac{3a+b}{4}\right) - K_2\left(\frac{3a+b}{4}\right) = \frac{32}{90}\cdot(b-a),\\
K_2\left(\frac{a+b}{2}\right) - K_3\left(\frac{a+b}{2}\right) = \frac{12}{90}\cdot(b-a),\\
K_3\left(\frac{a+3b}{4}\right) - K_4\left(\frac{a+3b}{4}\right) = \frac{32}{90}\cdot(b-a),\\
K_4(b) = \frac{7}{90}\cdot(b-a).
\end{array}\right.
\end{gather*}

\noindent
The above system is equivalent to
\begin{gather}
\left\{\begin{array}{l}
a+\beta_1  = -\frac{7}{90}\cdot(b-a),\\
\beta_1-\beta_2 = \frac{32}{90}\cdot(b-a),\\
\beta_2 -\beta_3 = \frac{12}{90}\cdot(b-a),\\
\beta_3 - \beta_4 = \frac{32}{90}\cdot(b-a),\\
b + \beta_4 = \frac{7}{90}\cdot(b-a).
\end{array}\right.
\label{betasystem}
\end{gather}

\noindent
The following Maple code:
\begin{verbatim}
beta1 := simplify(-7/90*(b-a)-a);
beta2 := simplify(beta1 - 32/90*(b-a));
beta3 := simplify(beta2-12/90*(b-a));
beta4 := simplify(beta3 - 32/90*(b-a));
\end{verbatim}

\noindent
returns the solution:
\begin{gather}
\left\{\begin{array}{l}
\beta_1 = -\frac{83a+7b}{90},\\
\beta_2 = -\frac{17a+13b}{30},\\
\beta_3 = -\frac{13a+17b}{30},\\
\beta_4 = -\frac{7a+83b}{90}.
\end{array}\right.
\label{solutionbeta}
\end{gather}

By the above choice of $\alpha_i,\ \beta_i$ and the fact that $K'(t)\equiv 1$, we have
$$\int_a^b\ K(t)f'(t)\ dt = \Boole(f) - \int_a^b\ f(t)\ dt.$$

\noindent
With the help of Maple software we check that $\int_a^b\ K(t)\ dt = 0:$
\begin{verbatim}
K1 := t-> alpha1*t + beta1;
K2 := t-> alpha2*t + beta2;
K3 := t-> alpha3*t + beta3;
K4 := t-> alpha4*t + beta4;
simplify(integrate(K1(t),t=a..(3*a+b)/4) 
+ integrate(K2(t),t=(3*a+b)/4..(a+b)/2)
+ integrate(K3(t),t=(a+b)/2..(a+3*b)/4) 
+ integrate(K4(t),t=(a+3*b)/4..b));
\end{verbatim}

\noindent
We have
\begin{equation}
\begin{split}
&\bigg|\Boole(f) - \int_a^b\ f(t)\ dt \bigg| \leq \bigg|\int_a^b\ K(t)f'(t)\ dt\bigg| = \bigg|\int_a^b\ K(t)\bigg(f'(t)-m(f')\bigg)\ dt\bigg|\\
&\leq \sup_{t\in[a,b]}\ |K(t)| \cdot \int_a^b\ |f'(t)-m(f')|\ dt = \sup_{t\in[a,b]}\ |K(t)| \cdot \bigg(I(f)-m(f')\bigg) (b-a).
\end{split}
\label{mainestimate1}
\end{equation}

It remains to estimate $\sup_{t\in[a,b]}\ |K(t)|$. Since every function $|K_i|,\ i=1,2,3,4$ is convex, when searching for the maximal value it suffices to check the enpoints of the subintervals. We obtain
\begin{gather*}
\sup_{t\in[a,\frac{3a+b}{4}]}\ |K_1(t)| = \frac{31}{180}\cdot (b-a) = \sup_{t\in[\frac{a+3b}{4},b]}\ |K_4(t)|,\\
\sup_{t\in[\frac{3a+b}{4},\frac{a+b}{2}]}\ |K_2(t)| = \frac{11}{60}\cdot (b-a) = \sup_{t\in[\frac{a+b}{2},\frac{a+3b}{4}]}\ |K_3(t)|,
\end{gather*}

\noindent
which may be verified with the Maple code:
\begin{verbatim}
simplify(max(abs(K1(a)),abs(K1((3*a+b)/4))));
simplify(max(abs(K2((3*a+b)/4)),abs(K2((a+b)/2))));
simplify(max(abs(K3((a+b)/2)),abs(K3((a+3*b)/4))));
simplify(max(abs(K4((a+3*b)/4)),abs(K4(b))));
\end{verbatim}

\noindent
We conclude that 
\begin{gather*}
\sup_{t\in[a,b]}\ |K(t)| = \frac{11}{60}\cdot (b-a),
\end{gather*}

\noindent
which, due to (\ref{mainestimate1}), proves (\ref{mainresult11}). Inequality (\ref{mainresult12}) is proven analogously.
\end{proof}

\begin{thm}
Let $f:[a,b]\longrightarrow \reals$ be such that $f'$ is absolutely continuous and $m(f'')$ and $M(f'')$ are both finite. Then 
\begin{gather}
\bigg|\int_a^b\ f(t)\ dt - \Boole(f)\bigg|\leq \frac{17}{1440}\cdot \bigg(I(f')-m(f'')\bigg) (b-a)^3
\label{mainresult21}
\end{gather}

\noindent
and
\begin{gather}
\bigg|\int_a^b\ f(t)\ dt - \Boole(f)\bigg|\leq \frac{17}{1440}\cdot \bigg(M(f'')-I(f')\bigg) (b-a)^3.
\label{mainresult22}
\end{gather}
\end{thm}
\begin{proof}
Let us define 
$$K_i(t) = \alpha_it^2 + \beta_it + \gamma_i, \hspace{0.4cm}\text{for}\hspace{0.4cm} i=1,2,3,4,$$

\noindent
where $\alpha_i,\ \beta_i,\ \gamma_i$ are constants which will be determined in the course of the proof. Again, we put
\begin{gather*}
K(t) = \left\{\begin{array}{ll}
K_1(t) & \text{if } t\in\left[a,\frac{3a+b}{4}\right],\\
K_2(t) & \text{if } t\in\left(\frac{3a+b}{4},\frac{a+b}{2}\right],\\
K_3(t) & \text{if } t\in\left(\frac{a+b}{2},\frac{a+3b}{4}\right],\\
K_4(t) & \text{if } t\in\left(\frac{a+3b}{4},b\right].
\end{array}\right.
\end{gather*}

Integration by parts yields
\begin{gather*}
\int_a^b\ K(t)f''(t)\ dt = C_1 + C_2 + \int_a^b\ K''(t)f(t)\ dt,
\end{gather*}

\noindent
where
\begin{equation*}
\begin{split}
C_1 = &-K_1(a)f'(a) + \left(K_1\left(\frac{3a+b}{4}\right) - K_2\left(\frac{3a+b}{4}\right)\right)f'\left(\frac{3a+b}{4}\right) \\
&+ \left(K_2\left(\frac{a+b}{2}\right) - K_3\left(\frac{a+b}{2}\right)\right)f'\left(\frac{a+b}{2}\right) \\
&+ \left(K_3\left(\frac{a+3b}{4}\right) - K_4\left(\frac{a+3b}{4}\right)\right)f'\left(\frac{a+3b}{4}\right) + K_4(b)f'(b),
\end{split}
\end{equation*}

\noindent
and
\begin{equation*}
\begin{split}
C_2 = &K_1'(a)f(a) - \left(K_1'\left(\frac{3a+b}{4}\right) - K_2'\left(\frac{3a+b}{4}\right)\right)f\left(\frac{3a+b}{4}\right) \\
&- \left(K_2'\left(\frac{a+b}{2}\right) - K_3'\left(\frac{a+b}{2}\right)\right)f\left(\frac{a+b}{2}\right) \\
&- \left(K_3'\left(\frac{a+3b}{4}\right) - K_4'\left(\frac{a+3b}{4}\right)\right)f\left(\frac{a+3b}{4}\right) - K_4'(b)f(b).
\end{split}
\end{equation*}

\noindent
We demand that $K''(t)\equiv 1$, so 
$$\alpha_1 = \alpha_2 = \alpha_3 = \alpha_4 = \frac{1}{2}.$$

\noindent
Furthermore, we impose the following conditions
\begin{gather*}
\left\{\begin{array}{l}
K_1'(a) = -\frac{7}{90}\cdot(b-a),\\
K_1'\left(\frac{3a+b}{4}\right) - K_2'\left(\frac{3a+b}{4}\right) = \frac{32}{90}\cdot(b-a),\\
K_2'\left(\frac{a+b}{2}\right) - K_3'\left(\frac{a+b}{2}\right) = \frac{12}{90}\cdot(b-a),\\
K_3'\left(\frac{a+3b}{4}\right) - K_4'\left(\frac{a+3b}{4}\right) = \frac{32}{90}\cdot(b-a),\\
K_4'(b) = \frac{7}{90}\cdot(b-a),
\end{array}\right.
\end{gather*}

\noindent
which are equivalent to the system (\ref{betasystem}). We already know that the solution to this system is given by (\ref{solutionbeta}).

Last but not least, we require that 
\begin{gather*}
\left\{\begin{array}{l}
K_1(a) = 0,\\
K_1\left(\frac{3a+b}{4}\right) - K_2\left(\frac{3a+b}{4}\right) = 0,\\
K_2\left(\frac{a+b}{2}\right) - K_3\left(\frac{a+b}{2}\right) = 0,\\
K_3\left(\frac{a+3b}{4}\right) - K_4\left(\frac{a+3b}{4}\right) = 0,\\
K_4(b) = 0.
\end{array}\right.
\end{gather*}

\noindent
This system is equivalent to:
\begin{gather}
\left\{\begin{array}{l}
\alpha_1 a^2 + \beta_1a + \gamma_1 = 0,\\
(\beta_1-\beta_2)\cdot \frac{3a+b}{4} + \gamma_1-\gamma_2 = 0,\\
(\beta_2-\beta_3)\cdot \frac{a+b}{2} + \gamma_2-\gamma_3 = 0,\\
(\beta_3-\beta_4)\cdot \frac{a+3b}{4} + \gamma_3-\gamma_4 = 0,\\
\alpha_4b^2+\beta_4b+\gamma_4 = 0.
\end{array}\right.
\label{gammasystem}
\end{gather}

\noindent
The following Maple code:
\begin{verbatim}
gamma1 := -simplify(alpha1*a^2+beta1*a);
gamma2 := simplify((beta1-beta2)*(3*a+b)/4+gamma1);
gamma3 := simplify((beta2-beta3)*(a+b)/2+gamma2);
gamma4 := simplify((beta3-beta4)*(a+3*b)/4+gamma3);
\end{verbatim}

\noindent
returns the solution:
\begin{gather}
\left\{\begin{array}{l}
\gamma_1 = \frac{19}{45}\cdot a^2 + \frac{7}{90}\cdot ab,\\
\gamma_2 = \frac{7}{45}\cdot a^2 + \frac{23}{90}\cdot ab + \frac{4}{45}\cdot b^2,\\
\gamma_3 = \frac{(a+2b)(8a+7b)}{90},\\
\gamma_4 = \frac{b(7a+38b)}{90}.
\end{array}\right.
\label{solutiongamma}
\end{gather}

By the above choice of $\alpha_i,\ \beta_i,\ \gamma_i$ and the fact that $K''(t)\equiv 1$, we have
$$\int_a^b\ K(t)f''(t)\ dt = \Boole(f) - \int_a^b\ f(t)\ dt.$$

\noindent
With the help of Maple software we check that $\int_a^b\ K(t)\ dt = 0:$
\begin{verbatim}
K1 := t-> alpha1*t^2 + beta1*t + gamma1;
K2 := t-> alpha2*t^2 + beta2*t + gamma2;
K3 := t-> alpha3*t^2 + beta3*t + gamma3;
K4 := t-> alpha4*t^2 + beta4*t + gamma4;
simplify(integrate(K1(t),t=a..(3*a+b)/4) 
+ integrate(K2(t),t=(3*a+b)/4..(a+b)/2)
+ integrate(K3(t),t=(a+b)/2..(a+3*b)/4) 
+ integrate(K4(t),t=(a+3*b)/4..b));
\end{verbatim}

\noindent
We have
\begin{equation}
\begin{split}
&\bigg|\Boole(f) - \int_a^b\ f(t)\ dt \bigg| \leq \bigg|\int_a^b\ K(t)f''(t)\ dt\bigg| = \bigg|\int_a^b\ K(t)\bigg(f''(t)-m(f'')\bigg)\ dt\bigg|\\
&\leq \sup_{t\in[a,b]}\ |K(t)| \cdot \int_a^b\ |f''(t)-m(f'')|\ dt = \sup_{t\in[a,b]}\ |K(t)| \cdot \bigg(I(f')-m(f'')\bigg) (b-a).
\end{split}
\label{mainestimate2}
\end{equation}

It remains to estimate $\sup_{t\in[a,b]}\ |K(t)|$. It is easy to see that the critical points of $K_i$ are $-\beta_i$ respectively. Furthermore, we have 
\begin{gather*}
a \leq -\beta_1 \leq \frac{3a+b}{4},\\
\frac{3a+b}{4} \leq -\beta_2 \leq \frac{a+b}{2},\\
\frac{a+b}{2} \leq -\beta_3 \leq \frac{a+3b}{4},\\
\frac{a+3b}{4} \leq -\beta_4 \leq b.
\end{gather*}

\noindent
Hence, we have
\begin{equation*}
\begin{split}
\sup_{t\in[a,\frac{3a+b}{4}]}\ |K_1(t)| &= \sup_{t\in[\frac{3a+b}{4},\frac{a+b}{2}]}\ |K_2(t)| = \sup_{t\in[\frac{a+b}{2},\frac{a+3b}{4}]}\ |K_3(t)| \\
&= \sup_{t\in[\frac{a+3b}{4},b]}\ |K_4(t)| = \frac{17}{1440}\cdot (b-a)^2,
\end{split}
\end{equation*}

\noindent
which can be verified with the following Maple code:
\begin{verbatim}
simplify(max(abs(K1(a)),abs(K1(-beta1)),abs(K1((3*a+b)/4))));
simplify(max(abs(K2((3*a+b)/4)),abs(K2(-beta2)),abs(K2((a+b)/2))));
simplify(max(abs(K3((a+b)/2)),abs(K3(-beta3)),abs(K3((a+3*b)/4))));
simplify(max(abs(K4((a+3*b)/4)),abs(K4(-beta4)),abs(K4(b))));
\end{verbatim}

\noindent
We conclude that 
$$\sup_{t\in[a,b]}\ |K(t)| = \frac{17}{1440}\cdot (b-a)^2,$$

\noindent
which, due to (\ref{mainestimate2}) proves (\ref{mainresult21}). Inequality (\ref{mainresult22}) is proven analogously. 
\end{proof}

\begin{thm}
Let $f:[a,b]\longrightarrow \reals$ be such that $f''$ is absolutely continuous and $m(f''')$ and $M(f''')$ are both finite. Then 
\begin{gather}
\bigg|\int_a^b\ f(t)\ dt - \Boole(f)\bigg|\leq \frac{1}{1620}\cdot \bigg(I(f'')-m(f''')\bigg) (b-a)^4
\label{mainresult31}
\end{gather}

\noindent
and
\begin{gather}
\bigg|\int_a^b\ f(t)\ dt - \Boole(f)\bigg|\leq \frac{1}{1620}\cdot \bigg(M(f''')-I(f'')\bigg) (b-a)^4.
\label{mainresult32}
\end{gather}
\end{thm}
\begin{proof}
Let us define 
$$K_i(t) = \alpha_it^3 + \beta_it^2 + \gamma_it + \delta_i, \hspace{0.4cm}\text{for}\hspace{0.4cm} i=1,2,3,4,$$

\noindent
where $\alpha_i,\ \beta_i,\ \gamma_i,\ \delta_i$ are constants which will be determined in the course of the proof. Again, we put
\begin{gather*}
K(t) = \left\{\begin{array}{ll}
K_1(t) & \text{if } t\in\left[a,\frac{3a+b}{4}\right],\\
K_2(t) & \text{if } t\in\left(\frac{3a+b}{4},\frac{a+b}{2}\right],\\
K_3(t) & \text{if } t\in\left(\frac{a+b}{2},\frac{a+3b}{4}\right],\\
K_4(t) & \text{if } t\in\left(\frac{a+3b}{4},b\right].
\end{array}\right.
\end{gather*}

Integration by parts yields
\begin{gather*}
\int_a^b\ K(t)f'''(t)\ dt = C_1 + C_2 + C_3 - \int_a^b\ K'''(t)f(t)\ dt,
\end{gather*}

\noindent
where
\begin{equation*}
\begin{split}
C_1 = &-K_1(a)f''(a) + \left(K_1\left(\frac{3a+b}{4}\right) - K_2\left(\frac{3a+b}{4}\right)\right)f''\left(\frac{3a+b}{4}\right) \\
&+ \left(K_2\left(\frac{a+b}{2}\right) - K_3\left(\frac{a+b}{2}\right)\right)f''\left(\frac{a+b}{2}\right) \\
&+ \left(K_3\left(\frac{a+3b}{4}\right) - K_4\left(\frac{a+3b}{4}\right)\right)f''\left(\frac{a+3b}{4}\right) + K_4(b)f''(b), \\
C_2 = &K_1'(a)f'(a) - \left(K_1'\left(\frac{3a+b}{4}\right) - K_2'\left(\frac{3a+b}{4}\right)\right)f'\left(\frac{3a+b}{4}\right) \\
&- \left(K_2'\left(\frac{a+b}{2}\right) - K_3'\left(\frac{a+b}{2}\right)\right)f'\left(\frac{a+b}{2}\right) \\
&- \left(K_3'\left(\frac{a+3b}{4}\right) - K_4'\left(\frac{a+3b}{4}\right)\right)f'\left(\frac{a+3b}{4}\right) - K_4'(b)f'(b),
\end{split}
\end{equation*}

\noindent
and
\begin{equation*}
\begin{split}
C_3 = &-K_1''(a)f(a) + \left(K_1''\left(\frac{3a+b}{4}\right) - K_2''\left(\frac{3a+b}{4}\right)\right)f\left(\frac{3a+b}{4}\right) \\
&+ \left(K_2''\left(\frac{a+b}{2}\right) - K_3''\left(\frac{a+b}{2}\right)\right)f\left(\frac{a+b}{2}\right) \\
&+ \left(K_3''\left(\frac{a+3b}{4}\right) - K_4''\left(\frac{a+3b}{4}\right)\right)f\left(\frac{a+3b}{4}\right) + K_4''(b)f(b).
\end{split}
\end{equation*}

\noindent
We demand that $K''(t)\equiv 1$, so 
$$\alpha_1 = \alpha_2 = \alpha_3 = \alpha_4 = \frac{1}{6}.$$

\noindent
Furthermore, we impose the following conditions
\begin{gather*}
\left\{\begin{array}{l}
K_1''(a) = -\frac{7}{90}\cdot(b-a),\\
K_1''\left(\frac{3a+b}{4}\right) - K_2'\left(\frac{3a+b}{4}\right) = \frac{32}{90}\cdot(b-a),\\
K_2''\left(\frac{a+b}{2}\right) - K_3'\left(\frac{a+b}{2}\right) = \frac{12}{90}\cdot(b-a),\\
K_3''\left(\frac{a+3b}{4}\right) - K_4'\left(\frac{a+3b}{4}\right) = \frac{32}{90}\cdot(b-a),\\
K_4''(b) = \frac{7}{90}\cdot(b-a),
\end{array}\right.
\end{gather*}

\noindent
which are equivalent to 
\begin{gather*}
\left\{\begin{array}{l}
a+2\beta_1  = -\frac{7}{90}\cdot(b-a),\\
2(\beta_1-\beta_2) = \frac{32}{90}\cdot(b-a),\\
2(\beta_2 -\beta_3) = \frac{12}{90}\cdot(b-a),\\
2(\beta_3 - \beta_4) = \frac{32}{90}\cdot(b-a),\\
b + 2\beta_4 = \frac{7}{90}\cdot(b-a).
\end{array}\right.
\end{gather*}

\noindent
The following Maple code:
\begin{verbatim}
beta1:= simplify((-a-7/90*(b-a))/2);
beta2:= simplify(beta1-16/90*(b-a));
beta3:= simplify(beta2-6/90*(b-a));
beta4:= simplify(beta3-16/90*(b-a));
\end{verbatim}

\noindent
returns the solution:
\begin{gather*}
\left\{\begin{array}{l}
\beta_1 = -\frac{83a+7b}{180},\\
\beta_2 = -\frac{17a+13b}{60},\\
\beta_3 = -\frac{13a+17b}{60},\\
\beta_4 = -\frac{7a+83b}{180}.
\end{array}\right.
\end{gather*}

What is more, we demand that 
\begin{gather*}
\left\{\begin{array}{l}
K_1'(a) = 0,\\
K_1'\left(\frac{3a+b}{4}\right) - K_2'\left(\frac{3a+b}{4}\right) = 0,\\
K_2'\left(\frac{a+b}{2}\right) - K_3'\left(\frac{a+b}{2}\right) = 0,\\
K_3'\left(\frac{a+3b}{4}\right) - K_4'\left(\frac{a+3b}{4}\right) = 0,\\
K_4'(b) = 0.
\end{array}\right.
\end{gather*}

\noindent
This system is equivalent to:
\begin{gather*}
\left\{\begin{array}{l}
3\alpha_1 a^2 + 2\beta_1a + \gamma_1 = 0,\\
(\beta_1-\beta_2)\cdot \frac{3a+b}{2} + \gamma_1-\gamma_2 = 0,\\
(\beta_2-\beta_3)\cdot (a+b) + \gamma_2-\gamma_3 = 0,\\
(\beta_3-\beta_4)\cdot \frac{a+3b}{2} + \gamma_3-\gamma_4 = 0,\\
3\alpha_4b^2+2\beta_4b+\gamma_4 = 0.
\end{array}\right.
\end{gather*}

\noindent
The following Maple code:
\begin{verbatim}
gamma1:= simplify(-3*alpha1*a^2-2*beta1*a);
gamma2:= simplify((3*a+b)/2*(beta1-beta2) + gamma1);
gamma3:= simplify((a+b)*(beta2-beta3) + gamma2);
gamma4:= simplify((a+3*b)/2*(beta3-beta4) + gamma3);
\end{verbatim}

\noindent
returns the solution:
\begin{gather*}
\left\{\begin{array}{l}
\gamma_1 = \frac{19}{45}\cdot a^2 + \frac{7}{90}\cdot ab,\\
\gamma_2 = \frac{7}{45}\cdot a^2 + \frac{23}{90}\cdot ab + \frac{4}{45}\cdot b^2,\\
\gamma_3 = \frac{(a+2b)(8a+7b)}{90},\\
\gamma_4 = \frac{b(7a+38b)}{90}.
\end{array}\right.
\end{gather*}

Finally, we require that 
\begin{gather*}
\left\{\begin{array}{l}
K_1(a) = 0,\\
K_1\left(\frac{3a+b}{4}\right) - K_2\left(\frac{3a+b}{4}\right) = 0,\\
K_2\left(\frac{a+b}{2}\right) - K_3\left(\frac{a+b}{2}\right) = 0,\\
K_3\left(\frac{a+3b}{4}\right) - K_4\left(\frac{a+3b}{4}\right) = 0,\\
K_4(b) = 0.
\end{array}\right.
\end{gather*}

\noindent
This system is equivalent to:
\begin{gather*}
\left\{\begin{array}{l}
\alpha_1 a^3 + \beta_1a^2 + \gamma_1a + \delta_1 = 0,\\
(\beta_1-\beta_2)\cdot \left(\frac{3a+b}{4}\right)^2 + (\gamma_1-\gamma_2)\cdot \frac{3a+b}{4} + \delta_1-\delta_2 = 0,\\
(\beta_2-\beta_3)\cdot \left(\frac{a+b}{2}\right)^2 + (\gamma_2-\gamma_3)\cdot \frac{a+b}{2} + \delta_2-\delta_3= 0,\\
(\beta_3-\beta_4)\cdot \left(\frac{a+3b}{4}\right)^2 + (\gamma_3-\gamma_4)\cdot \frac{a+3b}{4} + \delta_3-\delta_4 = 0,\\
\alpha_4b^3+\beta_4b^2+\gamma_4b+\delta_4 = 0.
\end{array}\right.
\end{gather*}

\noindent
The following Maple code:
\begin{verbatim}
delta1:= -simplify(alpha1*a^3+beta1*a^2+gamma1*a);
delta2:= simplify((beta1-beta2)*((3*a+b)/4)^2 
+ (gamma1-gamma2)*(3*a+b)/4 + delta1);
delta3:= simplify((beta2-beta3)*((a+b)/2)^2 
+ (gamma2-gamma3)*(a+b)/2 + delta2);
delta4:= simplify((beta3-beta4)*((a+3*b)/4)^2 
+ (gamma3-gamma4)*(a+3*b)/4 + delta3);
\end{verbatim}

\noindent
returns the solution:
\begin{gather*}
\left\{\begin{array}{l}
\delta_1 = -\frac{a^2(23a+7b)}{180},\\
\delta_2 = -\left(\frac{1}{36}\cdot a^3 + \frac{13}{180}\cdot a^2b + \frac{1}{18}\cdot ab^2 + \frac{1}{90}\cdot b^3\right),\\
\delta_3 = -\left(\frac{1}{90}\cdot a^3 + \frac{1}{18}\cdot a^2b + \frac{13}{180}\cdot ab^2 + \frac{1}{36}\cdot b^3\right),\\
\delta_4 = -\frac{b^2(7a+23b)}{180}.
\end{array}\right.
\end{gather*}

By the above choice of $\alpha_i,\ \beta_i,\ \gamma_i,\ \delta_i$ and the fact that $K'''(t)\equiv 1$, we have
$$\int_a^b\ K(t)f'''(t)\ dt = \Boole(f) - \int_a^b\ f(t)\ dt.$$

\noindent
With the help of Maple software we check that $\int_a^b\ K(t)\ dt = 0:$
\begin{verbatim}
K1 := t-> alpha1*t^3 + beta1*t^2 + gamma1*t + delta1;
K2 := t-> alpha2*t^3 + beta2*t^2 + gamma2*t + delta2;
K3 := t-> alpha3*t^3 + beta3*t^2 + gamma3*t + delta3;
K4 := t-> alpha4*t^3 + beta4*t^2 + gamma4*t + delta4;
simplify(integrate(K1(t),t=a..(3*a+b)/4) 
+ integrate(K2(t),t=(3*a+b)/4..(a+b)/2)
+ integrate(K3(t),t=(a+b)/2..(a+3*b)/4) 
+ integrate(K4(t),t=(a+3*b)/4..b));
\end{verbatim}

\noindent
We have
\begin{equation}
\begin{split}
&\bigg|\Boole(f) - \int_a^b\ f(t)\ dt \bigg| \leq \bigg|\int_a^b\ K(t)f'''(t)\ dt\bigg| = \bigg|\int_a^b\ K(t)\bigg(f'''(t)-m(f''')\bigg)\ dt\bigg|\\
&\leq \sup_{t\in[a,b]}\ |K(t)| \cdot \int_a^b\ |f'''(t)-m(f''')|\ dt = \sup_{t\in[a,b]}\ |K(t)| \cdot \bigg(I(f'')-m(f''')\bigg) (b-a).
\end{split}
\label{mainestimate3}
\end{equation}

It remains to estimate $\sup_{t\in[a,b]}\ |K(t)|$. At first we calculate critical points for $K_i:$
\begin{center}
\begin{tabular}{| c | c |}
    \hline
    \textbf{Function} & \textbf{Critical points}  \\ \hline
		$K_1$ & $a, \frac{38a+7b}{45}$ \\ \hline
		$K_2$ & $\frac{7a+8b}{15}, \frac{2a+b}{3}$ \\ \hline
		$K_3$ & $\frac{a+2b}{3}, \frac{8a+7b}{15}$ \\ \hline
		$K_4$ & $\frac{7a+38b}{45},b$ \\ \hline
\end{tabular}
\end{center}

\noindent
The above values are obtained with the Maple code:
\begin{verbatim}
solve(D(K1)(t)=0,t);
solve(D(K2)(t)=0,t);
solve(D(K3)(t)=0,t);
solve(D(K4)(t)=0,t);
\end{verbatim}

\noindent
However, we observe that 
\begin{gather*}
\frac{7a+8b}{15} \not\in \left[\frac{3a+b}{4},\frac{a+b}{2}\right],\\
\frac{8a+7b}{15} \not\in \left[\frac{a+b}{2},\frac{a+3b}{4}\right].
\end{gather*}

\noindent
The Maple code 
\begin{verbatim}
t1:=(38*a+7*b)/45;
t2:=(2*a+b)/3;
t3:=(a+2*b)/3;
t4:=(7*a+38*b)/45;

simplify(max(abs(K1(a)),abs(K1(t1)),abs(K1((3*a+b)/4))));
simplify(max(abs(K2((3*a+b)/4)),abs(K2(t2)),abs(K2((a+b)/2))));
simplify(max(abs(K3((a+b)/2)),abs(K3(t3)),abs(K3((a+3*b)/4))));
simplify(max(abs(K4((a+3*b)/4)),abs(K4(t4)),abs(K4(b))));
\end{verbatim}

\noindent
returns
\begin{equation*}
\begin{split}
\sup_{t\in[a,\frac{3a+b}{4}]}\ |K_1(t)| &= \frac{343}{1093500}\cdot (b-a)^3 = \sup_{t\in[\frac{a+3b}{4},b]}\ |K_4(t)|,\\
\sup_{t\in[\frac{3a+b}{4},\frac{a+b}{2}]}\ |K_2(t)| &= \frac{1}{1620}\cdot (b-a)^3 =  \sup_{t\in[\frac{a+b}{2}\frac{a+3b}{4}]}\ |K_3(t)|,
\end{split}
\end{equation*}

\noindent
We conclude that 
$$\sup_{t\in[a,b]}\ |K(t)| = \frac{1}{1620}\cdot (b-a)^3,$$

\noindent
which, due to (\ref{mainestimate3}) proves (\ref{mainresult31}). Inequality (\ref{mainresult32}) is proven analogously. 
\end{proof}

\section{Examples}
\label{examples}

Let us consider a monomial $f(t) = t^k,$ where the power $k$ is treated as a parameter. Let $a=0$ and $b>0$. The classical error estimate (comp. \cite{Atkinson}, \mbox{p. 266}) for Boole's rule is 
$$\frac{8}{945}\cdot h^7\|f^{(6)}\|,$$

\noindent
which in our case turns out to be 
\begin{gather}
\frac{b^{k+1}}{1935360}\cdot k(k-1)(k-2)(k-3)(k-4)(k-5).
\label{classicalestimate}
\end{gather}

Furthermore, we have
\begin{gather*}
I(f) = 1,\hspace{0.4cm} I(f') = k,\hspace{0.4cm} I(f'') = k(k-1),\\
m(f') = m(f'') = m(f''') = 0,\\
\end{gather*}

\noindent
and
\begin{gather*}
M(f') = kb^{k-1},\\
M(f'') = k(k-1)b^{k-2},\\
M(f''') = k(k-1)(k-2)b^{k-3}.
\end{gather*}

\noindent
In the table below, we summarize the results. The second column is the error bound of the estimate, i.e. the term on the right-hand side of the inequality. The last column indicates for which $k$ does the novel estimate perform better than (\ref{classicalestimate}).

\begin{center}
\begin{tabular}{| c | c | c |}
    \hline
    Estimate & Error bound & Better than (\ref{classicalestimate}) for $k\geq $ \\ \hline
		(\ref{mainresult11}) & $\frac{11}{6}\cdot b^{k+1}$ & 15\\ \hline
		(\ref{mainresult12}) & $\frac{11}{6}\cdot (k-1)b^{k+1}$ & 24\\ \hline
		(\ref{mainresult21}) & $\frac{17}{1440}\cdot kb^{k+1}$ & 11\\ \hline
		(\ref{mainresult22}) & $\frac{17}{1440}\cdot k(k-2)b^{k+1}$ & 16\\ \hline
		(\ref{mainresult31}) & $\frac{1}{1620}\cdot k(k-1)b^{k+1}$ & 10\\ \hline
		(\ref{mainresult32}) & $\frac{1}{1620}\cdot k(k-1)(k-3)b^{k+1}$ & 15\\ \hline
\end{tabular}
\end{center}



\begin{thebibliography}{00}

\bibitem[Agarwal R.P., Cerone P., Dragomir S.S. (2000)]{DragomirAgarwalCerone}
	Agarwal R.P., Cerone P., Dragomir S.S.: \textit{On Simpson’s inequality and applications}, Journal of Inequalities and Applications 5, p. $533-579$, (2000)
\bibitem[Atkinson K.E. (1989)]{Atkinson}
	Atkinson K.E.: \textit{An introdution to numerical analysis}, John Wiley and Sons, 1989
\bibitem[Cruz-Uribe D., Neugebauer C.J. (2002)]{CruzUribeNeugebauer}
	Cruz-Uribe D., Neugebauer C.J.: \textit{Sharp error bounds for the trapezoidal rule and Simpson’s rule}, Journal of Inequalities in Pure and Applied Mathematics 3, p. $1-22$, (2002)
\bibitem[Dragomir S.S. (1998)]{Dragomir}
	Dragomir S.S.: \textit{On Simpson's quadrature formula for differentiable mappings whose derivatives belong to $L^p$ spaces and applications}, Journal of Korean Society for Industrial and Applied Mathematics 2, p. $57-65$, (1998)
\bibitem[Dragomir S.S., Pecaric J., Wang S. (2000)]{DragomirPecaricWang}
	Dragomir S.S., Pecaric J., Wang S.: \textit{The unified treatment of trapezoid, Simpson and Ostrowski type inequalities for monotonic mappings and applications}, Mathematical and Computer Modelling 31, p. $61-70$, (2000)
\bibitem[Dragomir S.S., Fedotov I. (1999)]{DragomirFedotov}
	Dragomir S.S., Fedotov I.: \textit{An inequality of Ostrowski type and its applications for Simpson’s rule and special means}, Mathematical Inequalities and Applications 2, p. $491-499$, (1999)
\bibitem[Ujevic N. (2004)]{Ujevic2004}
	Ujevic N.: \textit{Sharp inequalities of Simpson type and Ostrowski type}, Computers and Mathematics with Applications 48, p. $145-151$, (2004)
\bibitem[Ujevic N. (2007)]{Ujevic2007}
	Ujevic N.: \textit{New error bounds for the Simpson’s quadrature rule and applications}, Computers and Mathematics with Applications 53, p. $64-72$, (2007)

\end{thebibliography}


\end{document}